\documentclass[twoside,11pt]{article}

\usepackage{blindtext}
\usepackage{siunitx}

\usepackage{amsmath,amssymb,amsfonts}
\usepackage{graphicx}
\usepackage{bm}
\usepackage{physics}
\usepackage{setspace}
\usepackage{xcolor}
\usepackage{amsthm}
\usepackage{graphicx, tikz}
\usepackage{mathtools}
\usetikzlibrary{3d, arrows.meta, positioning}
\usepackage{jmlr2e}

\newcommand{\R}{\mathbb{R}}

\newcommand{\identity}{\mathbb{I}_d}

\usepackage{lastpage}
\jmlrheading{}{}{}{}{}{}{}

\ShortHeadings{Structured Drift Design}{Taheri}
\firstpageno{1}

\begin{document}

\title{Structured Drift Design for Denoising Diffusion Models}

\author{\name Mahsa Taheri \email mahsa.taheri@uni-hamburg.de \\
       \addr Department of Mathematics\\
       University of Hamburg
       }

\editor{}

\maketitle

\begin{abstract}
Diffusion-based generative models have achieved remarkable success in high-dimensional data generation; however, they  fundamentally rely on isotropic diffusion processes that destroy meaningful geometric structures in the forward process. For complex, multimodal, and highly correlated distributions—such as biologically constrained genetic data—isotropic noise merges distinct modes and distorts intrinsic dependencies. This forces the reverse process to recover structure from heavily degraded signals, leading to slow convergence, mode averaging, and biologically implausible samples.
To address this, we introduce the Geometry-aware Ornstein–Uhlenbeck (GOU) process, a structured drift design that embeds data geometry into forward and backward dynamics. By employing a variance-aware anisotropic drift, GOU contracts low-variance directions rapidly while preserving high-variance directions longer, maintaining key multimodal structures as stable channels over time. 
We provide rigorous theoretical analysis of the proposed dynamics, including the evolution of moments in the forward process and bounds on the initialization error of the backward process under GOU. 
Our analysis shows that GOU reduce the initial mismatch for reverse-time sampling while preserving cluster-level structure, leading to faster and more accurate generation.
Experiments on synthetic and real-world biological data demonstrate improved structural preservation and sample quality with GOU.
\end{abstract}


\section{Introduction}
Diffusion models have achieved remarkable success in learning complex high-dimensional data distributions~\citep{song2019generative,ho2020denoising,song2020score,taheri2025regularization}. 
However, standard formulations rely on isotropic Gaussian perturbations to transform data into pure noise and a corresponding reverse process to reconstruct samples from the noise distribution.
Both processes treat all directions equally and ignore the underlying geometric and correlation structure present in many scientific datasets.
In the forward process, isotropy can blur distinct modes and destroy meaningful dependencies, limiting interpretability.
In the backward process, this can lead to inefficient sample generation, with convergence rates that depend on properties of the target distribution (e.g., log-concavity and  smoothness).
A crucial question araises here: What if we consider the geometry of data along the diffusion process?

Isotropic diffusion treats all directions uniformly, causing distinct modes in multimodal distributions to progressively blur together in the forward process.
The reverse process must then recover separated components from severely degraded signals, which is challenging both theoretically and empirically.
Existing convergence analyses show that diffusion samplers might mix slowly in non-log-concave landscapes, particularly in regions of low density between modes~\citep{block2020generative,tzen2019nonconvex,eberle2016reflection,lee2023convergence,gao2025wasserstein,kremling2025non}. 

Moreover, accurately recovering the correlation structure of data is crucial in applications such as genetics and biology, where covariance patterns—particularly linkage disequilibrium (LD)—capture functional, biological, and evolutionary constraints. Even small deviations from these dependency structures can result in samples that are biologically implausible or statistically invalid. If the forward diffusion process substantially disrupts these correlations, the reverse process may be unable to recover them accurately, potentially leading to unrealistic generated genotypes~\citep{lampis2026snpgen,sadia2025causalgendiff,si2026flag}.
While such sensitivity may be less pronounced in comparatively simple domains such as natural images, it becomes crucial in scientific applications where fine-scale dependency structures carry essential information. 
Recent works in genomic and single-cell generative modeling have highlighted that accurately preserving these correlations is necessary for generating realistic and scientifically meaningful samples~\citep{kenneweg2025geneticdiffusion,li2025genomicdiffusion}.

These limitations motivate diffusion processes that explicitly incorporate the geometry of the data into the diffusion dynamics.
Diffusion on manifolds replaces Euclidean noise with geometry-aware operators~\citep{de2022riemannian,huang2022riemannian}, but typically requires explicit knowledge of the underlying manifold and introduces substantial analytical and computational complexity. Preconditioned stochastic dynamics adapt sampling procedures using local geometry~\citep{girolami2011riemann,ma2015complete}, yet these approaches mainly target the backward process. 
Transport-based approaches, including neural ODEs and flow matching~\citep{chen2018neural,lipman2023flow}, learn flexible mappings between distributions, but do not provide a tractable or interpretable mechanism to encode geometry directly into the diffusion process.

Taken together, existing approaches either analyze diffusion models under fixed, isotropic dynamics or incorporate geometry in noise or in ways that are difficult to integrate into a tractable generative framework. As a result, the role of the forward diffusion process as a mechanism for encoding data-dependent structure remains largely unexplored.
Building on this perspective, we propose a structured diffusion framework that introduces anisotropic  feature-dependent drifts,  maintaining mode separation and preserving correlations among features for significantly longer times. 

\paragraph{Contributions}
The primary contributions of this work are summarized as follows:
\begin{enumerate}
    \item  We introduce  \emph{Geometry-aware Ornstein--Uhlenbeck (GOU)} process, a novel structure-preserving diffusion framework explicitly designed to mitigate the rapid decay of feature correlations and delay mode merging during the forward process through a structured, anisotropic drift.
    \item  We provide a rigorous \emph{theoretical analysis} mapping the exact analytical evolution of moments (means and covariances) under the proposed GOU dynamics.
        \item We analyze the stationary distribution induced by GOU and prove that the initialization error of the reverse process can be controlled by local variance. In particular, the mismatch is determined by within-cluster variability rather than the global variance of the data, thereby reducing the effective initialization error and accelerating the convergence of diffusion-based samplers.

    \item  We demonstrate empirically that GOU significantly improves \emph{mode separation, latent interpretability, and sample fidelity} over standard diffusion baselines across both synthetic benchmarks and real-world biologically structured datasets.
\end{enumerate}

\paragraph{Paper outline}
We review the related literature in Section~\ref{sec:RW}. 
In Section~\ref{sec:GOU}, we introduce the proposed \emph{geometry-aware Ornstein--Uhlenbeck (GOU)} process and provide an initial analysis of its ability to preserve mode separation during the forward diffusion process. 
In Section~\ref{sec:meancovev}, we present a theoretical analysis of the evolution of the mean and covariance under the GOU process, highlighting its structure-preserving properties. 
In Section~\ref{sec:backproc}, we study the reverse process of the GOU framework and analyze the associated initialization error. 
In Section~\ref{sec:disc}, we discuss the practical implications and advantages of the proposed framework compared with standard isotropic diffusion models. 
Finally, in Section~\ref{sec:sim}, we empirically evaluate our approach on synthetic and real datasets, demonstrating improved correlation preservation, mode separation, and sample fidelity compared to standard isotropic diffusion models.

\section{Related works}\label{sec:RW}
A growing body of research provides theoretical guarantees for diffusion models and their sampling procedures. Existing analyses establish convergence under various assumptions on the data distribution, score function, and discretization schemes~\citep{block2020generative,de2022convergence,wibisono2022convergence,chen2022sampling,chen2023improved,lee2023convergence,gentiloni2025beyond,liang2024non,gao2024convergence,gao2025wasserstein,kremling2025non,taheri2025regularization,lee2022convergence,chen2023probability,conforti2025kl,beyler2025convergence}. 
These results cover convergence in total variation, Kullback--Leibler divergence, and Wasserstein distances. 
Despite these advances, most analyses in Wasserstein distance rely on strong global assumptions on the target density, such as log-concavity, as well as regularity conditions on the score function, including Lipschitz continuity and smoothness.
More recent works aim to relax these assumptions by introducing weaker conditions, such as weak log-concavity or dissipativity~\citep{gentiloni2025beyond,bruno2025wasserstein,brigati2024heat,ishige2024eventual,kremling2025non}. 
While these approaches broaden the theoretical scope, they still fall short of covering realistic data settings, such as non-sub-Gaussian or heavy-tailed distributions, and often yield conservative (i.e., loose) convergence rates.
Moreover, these works  focus primarily on the reverse-time dynamics, leaving the forward process largely unexplored as a design component.

In a different direction, several works aim to incorporate geometric structure into generative modeling. Diffusion on manifolds replaces Euclidean noise with geometry-aware operators~\citep{de2022riemannian,huang2022riemannian}, but typically requires explicit knowledge of the underlying manifold and introduces substantial analytical and computational complexity. 
Preconditioned stochastic dynamics adapt sampling procedures using local geometry~\citep{girolami2011riemann,ma2015complete}, yet these approaches are designed to target just the backward process. 
Transport-based approaches, including neural ODEs and flow matching~\citep{chen2018neural,lipman2023flow}, learn flexible mappings between distributions, but do not provide a tractable or interpretable mechanism to encode geometry directly into the diffusion process.


A recent line of research (empirically) investigates  non-isotropic diffusion processes employing structured diffusion coefficients, including learning matrix-valued diffusion coefficients~\citep{liu2026variational}, edge-aware perturbations~\citep{vandersanden2024edge},  blue-noise masks~\citep{huang2024blue}, covariance-informed diffusion~\citep{huk2025diffusion}, and subspace- or frequency-restricted dynamics~\citep{jing2022subspace}, all of which introduce anisotropic noise to better match the data geometry.
In contrast, our approach formulates anisotropy through a generalized Ornstein–Uhlenbeck dynamics that explicitly encodes geometric structure in the drift function rather than  in the noise, providing a unified perspective on how geometry governs diffusion.
Despite these differences in modeling choices and application domains, these works collectively underscore the importance of incorporating data geometry into diffusion-based modeling, motivating our present study, which aims to integrate geometric structure directly into the forward dynamics in a principled and analytically tractable manner.

\paragraph{Notations}
For a vector $X\in \mathbb{R}^d$, we denote by $\norm{X}$ its Euclidean norm.
For any two probability measures $\mu, \nu \in \mathcal P_2(\mathbb R^d)$, the space of measures on $\mathbb R^d$ with finite second moment, the 2-Wasserstein distance, based on the Euclidean norm, is defined as 
\begin{equation*}\label{eq:wass-def}
    \mathcal W_2(\mu, \nu) :=
\left ( \inf_{ X \sim \mu, Y \sim \nu   } \mathbb E \|X - Y\|^2 \right ) ^{\frac 12}\,,
\end{equation*}
where the infimum is taken over all possible couplings of $\mu$ and $\nu$. 

\section{Geometry-aware diffusion models}\label{sec:GOU}

Standard denoising diffusion models use a forward Stochastic Differential Equation (SDE) that gradually turns data into noise, together with a reverse-time SDE or ODE (Ordinary Differential Equation) that reconstructs the data using a learned score function~\citep{song2020score}. While this framework performs well for generic data such as natural images, it ignores the underlying geometry of the data, including multimodality and the intrinsic dependency structure that characterizes many scientific datasets. Examples include correlations among genes, single nucleotide polymorphisms (SNPs), and other physical or biological measurements, where preserving these structures is essential for generating realistic samples. In particular, the forward process is typically isotropic, meaning that it treats all directions in the data space equally.
For example variance-preserving (VP) SDE is of the form
\begin{equation}\label{eq:isotropic-sde}
    \dd X_t = -\frac{1}{2}\beta(t) X_t \, \dd t + \sqrt{\beta(t)}\, \dd W_t,
\end{equation}
where $X_t \in \mathbb{R}^d$, $W_t$ is a $d$-dimensional Brownian motion, and $\beta(t) > 0$ is a time dependent drift function.

In this work, we propose  anisotropic stochastic processes, in which data are diffused respecting their underlying geometric structure. 
Our primary contribution is a forward diffusion process that mitigates mode mixing in multimodal distributions and preserves the intrinsic geometric structure of the data over extended diffusion times.

\subsection{Geometry-aware Ornstein-Uhlenbeck}

We design a drift matrix $M\in \R^{d \times d}$  for the Ornstein-Uhlenbeck (OU) process that contracts directions proportionally to the variance of data, allowing for controlled diffusion. Let $\Sigma \coloneqq \operatorname{Cov}_{p_0}(X)$ denote the covariance matrix of the target distribution $p_0$, with eigen-decomposition
\[
\Sigma =  U \Lambda U^\top\,, \qquad \Lambda = \operatorname{diag}(\lambda_1, \dots, \lambda_d)\,,
\]
where $U$ contains the principal directions  with $UU^T=\identity$ and $\{\lambda_i\}_{i=1}^d\ge 0$ are the corresponding variances along the principal directions. 

We propose  variance-aware contraction rate along each principal direction
\begin{equation}\label{eq:contraction-rate}
c_i := -b \left( \gamma + (1-\gamma) \left( 1 - \frac{\lambda_i}{\lambda_{\max}} \right) \right)\,, 
\qquad b, \gamma \ge 0\,, 
\end{equation}
where  $\lambda_{\max}$ is the largest eigenvalue and $b$ and $\gamma$ are hyperparameters to be chosen. 
This implies $c_i<0$ for all $i \in \{1,\dots,d\}$. 
Equation~\eqref{eq:contraction-rate} implies directions of high variance to contract slowly, whereas low-variance directions contract quickly:
\[
\lambda_i \text{ large} \implies |c_i| \text{ small (slow contraction)}\,, \qquad
\lambda_i \text{ small} \implies |c_i| \text{ large (fast contraction)}\,.
\]

We then employ the drift matrix 
\begin{equation}\label{eq:drift-matrix}
M := U \operatorname{diag}(c_1, \dots, c_d)\,U^\top\,,
\end{equation}
which is fixed and shared across the whole diffusion process. Employing the designed drift matrix $M$, we define the geometry-aware Ornstein-Uhlenbeck as
\begin{equation}\label{eq:structured-sde}
\mathrm{d}X_t = M (X_t - \mu) \, \mathrm{d}t + \sqrt{\beta} \, \mathrm{d}W_t\,,
\end{equation}
where $\mu\in \R^d$ may be chosen fixed as a global mean or another reference center and $\beta>0$ is a constant  diffusion coefficient.
In fact, the diffusion adapts the contraction rates to the estimated covariance geometry. Directions with small empirical variance are contracted more strongly, reducing the influence of poorly represented directions while preserving slower dynamics along dominant principal components.

The following immediate remark follows directly from~\eqref{eq:structured-sde}:

\begin{remark}[Mode merging]
The GOU process introduced in~\eqref{eq:structured-sde} contracts low-variance directions rapidly, causing modes to merge along these features, while the high-variance directions—where modes are naturally separated—decay slowly. 
\end{remark}




\subsection{Mode separation under GOU process}

In this section  we  study how local modes are preserved under the GOU forward process introduced in~\eqref{eq:structured-sde}. 
For simplicity, we consider a bimodal data distribution, although the results extend naturally to general multimodal settings.

\begin{theorem}[Mode separation under GOU]\label{thm:modesep}
Let $\mu_1(0), \mu_2(0) \in \mathbb{R}^d$ denote the initial means of two modes with $X_0$ has finite first and second moments. Let $\mu_k(t) = \mathbb{E}[X_t^{(k)}]$ be the corresponding mean trajectories under \eqref{eq:structured-sde}, initialized at $\mu_k(0)$ for $k \in \{1,2\}$. Define the initial separation of two modes
\[
\Delta_0 \coloneqq \mu_1(0) - \mu_2(0)
\]
and let define $U^\top \Delta_0 =: \alpha$, with components $\alpha_i$ for $i \in \{1,\dots,d\}$, which are coordinates in the rotated system.

Then, for all $t \ge 0$,
\[
\Delta(t) \coloneqq \mu_1(t) - \mu_2(t)
= e^{Mt} \Delta_0
= U \operatorname{diag}(e^{c_1 t}, \dots, e^{c_d t}) \,\alpha
= \sum_{i=1}^d \alpha_i e^{c_i t} u_i\,,
\]
where $u_i$ denotes the $i$-th column of $U$.

Moreover,
\[
\|\Delta(t)\|^2 = \sum_{i=1}^d \alpha_i^2 e^{2 c_i t}\,.
\]

In particular, each principal-direction component evolves independently with the rate~$e^{c_i t}$.
\end{theorem}

\begin{proof}
We study the evolution of the mean of the solution $X_t$ of the SDE defined in~\eqref{eq:structured-sde} in 4 steps: 

\emph{Step 1 (ODE for the mean)}:  
Consider the SDE
\[
\dd X_t = M(X_t-\mu)\, \dd t + \sqrt{\beta}\,  \dd W_t, \qquad X_0 \text{ has finite first moment.}
\]
Integrating both sides from $0$ to $t$ yields
\[
\int_0^t \dd X_s = \int_0^t M(X_s-\mu)\, \dd s + \sqrt{\beta} \int_0^t \dd W_s\,.
\]
Using the identities (recall $W_0=0$)
\[
\int_0^t \dd X_s = X_t - X_0\,, 
\qquad 
\int_0^t \dd W_s = W_t\,,
\]
we obtain the integral representation
\[
X_t - X_0 = \int_0^t M(X_s-\mu)\, \dd s + \sqrt{\beta}\, W_t\,.
\]

Taking expectations from both sides,  using linearity of expectation, together with $\mathbb{E}[W_t] = 0$, we obtain
\[
\mathbb{E}[X_t] - \mathbb{E}[X_0] = \int_0^t M (\mathbb{E}[X_s] - \mu)\, \dd s\,.
\]

The interchange of expectation and integration is justified by Fubini’s theorem since
\[
\int_0^t \mathbb{E}\big[|X_s|\big]\,  \dd s < \infty\,.
\]
Differentiating both sides with respect to $t$ then gives the ODE for the mean:
\[
\frac{\dd}{\dd t} \mathbb{E}[X_t] = M (\mathbb{E}[X_t] - \mu)\,, \qquad \mathbb{E}[X_0] = \mu_0\,.
\]

\emph{Step 2 (Difference of mean of two solutions)}: Consider two solutions of the SDE with initial means $\mu_1(0)$ and $\mu_2(0)$, denoted by $X_t^{(1)}$ and $X_t^{(2)}$. Let the mean difference be $\Delta(t) := \mathbb{E}[X_t^{(1)}] - \mathbb{E}[X_t^{(2)}]$. Each mean satisfies
\[
\frac{\dd}{ \dd t}\mathbb{E}[X_t^{(1)}] = M(\mathbb{E}[X_t^{(1)}]-\mu)\,, \qquad
\frac{\dd}{ \dd t}\mathbb{E}[X_t^{(2)}] = M(\mathbb{E}[X_t^{(2)}]-\mu)\,.
\]
Subtracting gives a linear ODE for the mean difference:
\[
\frac{\dd}{\dd t}\Delta(t) = M \Delta(t), \qquad \Delta(0) = \mu_1(0)-\mu_2(0) = \Delta_0\,.
\]

\emph{Step 3 (Solution of the linear ODE)}: The solution of this linear ODE with constant coefficients is
\[
\Delta(t) = e^{M t} \Delta_0\,.
\]

\emph{Step 4 (Componentwise representation)}: Diagonalize $M = U C U^\top$, where $C = \operatorname{diag}(c_1,\dots,c_d)$ and $U$ is an orthogonal matrix of eigenvectors. Using Lemma~\ref{lem:Mex} in the Appendix, 
\[
\Delta(t) = U e^{C t} U^\top \Delta_0
\]
and the squared norm can be written as
\[
\|\Delta(t)\|^2 = \| U e^{C t} U^\top \Delta_0 \|^2
= \sum_{i=1}^d \alpha_i^2 e^{2 c_i t}\,.
\]

This explicitly shows how each component of the initial mean difference evolves along the eigen-directions of $M$, concluding the proof.
\end{proof}

Theorem~\ref{thm:modesep} establishes that the separation between modes in high-variance directions decays significantly more slowly than in low-variance directions. We next state an immediate consequence of Theorem~\ref{thm:modesep}.

\begin{corollary}[Persistence of mode separation under GOU]\label{cor:modesep}
Let $\Delta_0$ denote the initial mean difference between two modes, and write $\alpha = U^\top \Delta_0$. 

If there exists at least one index $i$ such that $\alpha_i \neq 0$ and $|c_i|$ is small compared to the other $|c_j|$, then the corresponding component of the separation $\Delta(t)$ decays more slowly and thus persists for a longer time.

In particular, directions associated with small $|c_i|$ (high-variance directions) act as channels along which the modes remain separated at long times.
\end{corollary}

For a visual illustration, we refer to Figure~\ref{fig:mode-separation}, which shows that the GOU forward process preserves mode separation along the $X$-axis, whereas isotropic diffusion rapidly mixes samples from the two modes in a toy mixture of two Gaussians.

\begin{figure}
    \centering
    \includegraphics[width=0.7\linewidth]{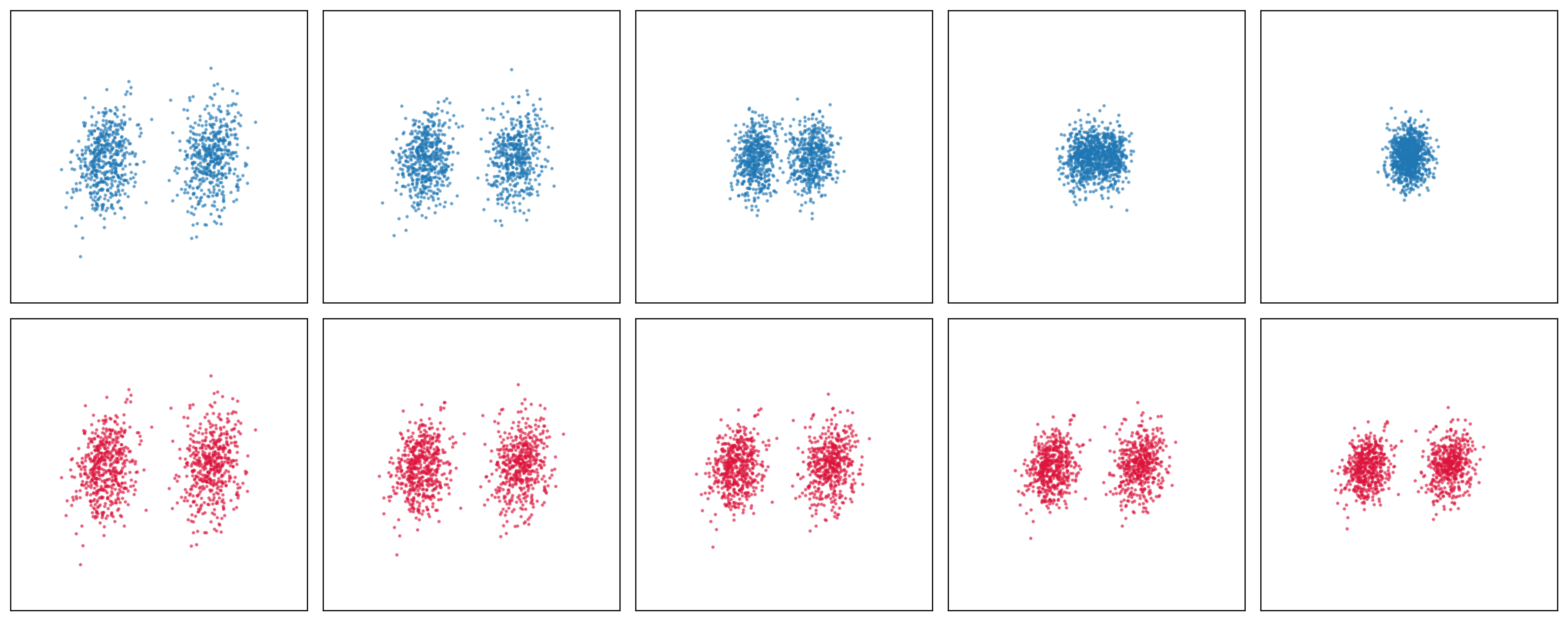}
    \caption{Comparison of the   isotropic diffusion (first row) with GOU process (second row) in the forward dynamics (for $t\in \{0,0.2,0.5,0.7,1\}$ from left to right). The GOU process maintains separation between modes along the $X$-axis, while isotropic diffusion quickly mixes samples across modes.}
    \label{fig:mode-separation}
\end{figure}

\section{Forward process under GOU}\label{sec:meancovev}

In this section, we analyze the evolution of the mean and covariance matrix under the GOU forward process and subsequently characterize its stationary distribution.

\subsection{Mean and covariance evolution under GOU}

We now turn  to study how the mean and covariance matrix diffuse under the GOU forward process.
We study the mean evolution in Theorem~\ref{thm:mean} and the covariance evolution in Theorem~\ref{thm:covariance}. 

\begin{theorem}[Mean evolution under GOU]\label{thm:mean}
Let $\mu_t := \mathbb{E}[X_t]$ and assume $X_0$ has finite first and second moments.  Then
\begin{equation}\label{eq:mean-solution}
    \mu_t = \mu + e^{Mt}(\mu_0 - \mu)\,.
\end{equation}
Equivalently, writing the spectral decomposition of $M$,
we obtain
\begin{equation}\label{eq:means}
   \mu_t = \mu + U \operatorname{diag}(e^{c_1 t}, \dots, e^{c_d t}) \,U^\top (\mu_0 - \mu)\,. 
\end{equation}

\end{theorem}

\begin{proof}
    Taking expectations on both sides of the SDE \eqref{eq:structured-sde} yields
\[
\frac{\dd}{\dd t}\mathbb{E}[X_t]
= M(\mathbb{E}[X_t]-\mu)\,,
\]
where we used $\mathbb{E}[\dd W_t]=0$ (see the proof of Theorem~\ref{thm:modesep} for further details).

For $\mu_t = \mathbb{E}[X_t]$, it satisfies the linear ODE
\begin{equation*}\label{eq:mean-ode}
    \frac{\dd \mu_t}{\dd t} = M(\mu_t-\mu)\,,
    \qquad \mu_0 = \mathbb{E}[X_0]\,.
\end{equation*}

Define $\tilde\mu_t := \mu_t - \mu$. Then $\tilde\mu_t$ satisfies
\[
\frac{\dd{\tilde\mu}_t}{\dd t} = M\tilde\mu_t,
\qquad \tilde\mu_0 = \mu_0 - \mu=0\,.
\]
The unique solution is given by
\[
\tilde\mu_t = e^{Mt}(\mu_0-\mu)\,.
\]
Therefore,
\[
\mu_t = \mu + e^{Mt}(\mu_0-\mu)\,,
\]
which proves \eqref{eq:mean-solution}.

Employing Lemma~\ref{lem:Mex}
gives~\eqref{eq:means}, as desired.  

\end{proof}

The explicit solution \eqref{eq:mean-solution} shows that the mean of the forward process of GOU evolves deterministically under the linear flow induced by the drift matrix $M$, independently of the noise level $\beta$. In contrast to isotropic diffusion, where the mean decays uniformly, the structured drift induces direction-dependent contraction toward the reference center $\mu$. For $M$ diagonalizable (as we defined), each principal component of the mean evolves independently at a rate determined by the corresponding eigenvalue of $M$, allowing  directions with high variance  to persist for longer time steps (see Figure~\ref{fig:mode-separation} for  a visual illustration).






\begin{theorem}[Covariance evolution under GOU]\label{thm:covariance}
Let $\Sigma_t := \operatorname{Cov}(X_t)$. Then $\Sigma_t$ satisfies the Lyapunov
differential equation
\begin{equation}\label{eq:lyapunov}
    \frac{\dd \Sigma_t}{\dd t}
    = M\Sigma_t + \Sigma_t M^\top + \beta \identity\,,
    \qquad \Sigma_0 = \operatorname{Cov}(X_0)\,.
\end{equation}
Moreover, the unique solution is given by
\begin{equation}\label{eq:covariance-solution}
    \Sigma_t
    = e^{Mt}\Sigma_0 e^{M^\top t}
    + \beta \int_0^t
    e^{M(t-s)} e^{M^\top (t-s)}\,\dd s\,.
\end{equation}
\end{theorem}

Theorem~\ref{thm:covariance} describes how the uncertainty of the diffusion process evolves over time. The covariance matrix $\Sigma_t$ changes through two mechanisms. First, the existing covariance $\Sigma_0$ is transported by the linear dynamics induced by the drift matrix $M$, which reshapes the correlations in the data according to the geometric structure encoded in $M$. Second, stochastic noise continuously injects additional variance into the system at rate $\beta$. The resulting covariance therefore consists of two parts: a transformed version of the initial covariance (first term on the right-hand side of~\eqref{eq:covariance-solution}) and an accumulated noise contribution (second term on the right-hand side of~\eqref{eq:covariance-solution}). Although the injected noise is isotropic, the drift matrix $M$ steers its propagation so that the evolving covariance remains anisotropic and aligned with the principal directions of the data.

Note also that the drift matrix $M$ appears in the diffusion part of the covariance evolution because the same linear dynamics that contract the mean also act on past noise increments. At each infinitesimal step, isotropic noise $\sqrt{\beta} \, \dd W_t$ is injected equally in all directions. However, once injected, each noise increment is subsequently transported by the drift propagator $e^{M(t-s)}$ from its injection time $s$ to the present time $t$. Directions with slow contraction (small $|c_i|$) retain past noise for longer, leading to larger accumulated variance, while fast-contracting directions (large $|c_i|$) quickly suppress past noise. Thus, although the instantaneous noise injection is isotropic, the total accumulated covariance $\beta \int_0^t e^{M(t-s)} e^{M^\top(t-s)} \, \dd s$ becomes anisotropic precisely because $M$ controls the memory of past fluctuations.
The proof of the Theorem~\ref{thm:covariance} is defered to the Appendix~\ref{sec:poofcov}. 

\subsection{Stationary distribution}

A direct consequence of Theorem~\ref{thm:mean} yields the mean of the stationary distribution in the limit as $t \to \infty$.

\begin{corollary}[Stationary mean of the GOU process]\label{corl:Stamean}
Under the same assumptions as in Theorem~\ref{thm:mean}, the mean $u_t$ converges as $t \to \infty$ to $u$.
\end{corollary}

Employing Theorem~\ref{thm:covariance}, we can directly compute the stationary covariance 
in the limit $t \to \infty$.

\begin{lemma}[Stationary covariance of the GOU process]\label{th:SCov}
Under the same assumptions as Theorem~\ref{thm:covariance}, the covariance matrix $\Sigma_t$ converges as $t \to \infty$ to a unique stationary covariance $\Sigma_\infty$, given by the unique solution of the Lyapunov equation
\begin{equation}
M \Sigma_\infty + \Sigma_\infty M^\top + \beta \identity = 0\,.
\end{equation}

Moreover, as $M$ is symmetric,  this simplifies to
\begin{equation}
\Sigma_\infty = -\frac{\beta}{2} M^{-1}\,.
\end{equation}
\end{lemma}

The proof of Lemma~\ref{th:SCov} is deferred to the Appendix~\ref{sec:proofSCov}.

\begin{remark}[Interpretation of the stationary covariance]
Lemma~\ref{th:SCov} shows that the stationary covariance is fully determined by the drift matrix $M$ and the noise level $\beta$.

\begin{itemize}

    \item \emph{Positivity:} Since $c_i < 0$, each variance component is strictly positive, as required for a valid covariance matrix.

    \item \emph{Relation to contraction rates:} From the definition of $c_i$ in Eq.~\eqref{eq:contraction-rate}, directions with weaker contraction (i.e., smaller $|c_i|$) correspond to larger stationary variance, while strongly contracting directions lead to smaller stationary variance.
\end{itemize}

Overall, the stationary distribution is anisotropic and encodes the geometry induced by the drift matrix $M$, in contrast to standard isotropic diffusion models.
\end{remark}

\begin{corollary}[Isotropic diffusion]
For $M = -\mathbb{I}_d/2$ and $\beta = 1$,  the stationary covariance becomes
\[
\Sigma_\infty = -\frac{1}{2} M^{-1} = -\frac{1}{2} \left(-\frac{1}{2} \mathbb{I}_d\right)^{-1}
= -\frac{1}{2} (-2 \mathbb{I}_d) = \mathbb{I}_d\,.
\]

Thus, in this case the stationary distribution is standard isotropic Gaussian, recovering classical Brownian motion-driven Ornstein–Uhlenbeck dynamics.
\end{corollary}

\section{Backward process under GOU}\label{sec:backproc}

We now move to study the reverse-time dynamics of the GOU process in~\eqref{eq:structured-sde}. These dynamics define the generative mechanism that transports samples from the stationary  distribution back to the target distribution while respecting the geometry encoded by the drift matrix $M$.

\subsection{Reverse-time SDE}

Consider the forward GOU process defined in \eqref{eq:structured-sde}, and let $p_t(x)$ denote the marginal density of $X_t$ at time $t$. 
From the general time-reversal theory for diffusion processes \citep{anderson1982reverse}, the reverse-time process $\{\bar X_t\}_{0\le t\le T}$, defined such that $\bar X_t = X_{T-t}$ in distribution, satisfies the stochastic differential equation
\begin{equation}\label{eq:reverse-sde-gou}
\dd \bar{X}_t
=
\big[
- M (\bar{X}_t - \mu)
+ \beta \nabla_x \log p_{T-t}(\bar{X}_t)
\big] \dd t
+
\sqrt{\beta}\, \dd \bar{W}_t \,,
\end{equation}
where $\bar W_t$ is a $d$-dimensional Brownian motion.
The reverse dynamics consist of two components. The first term, $-M(\bar X_t-\mu)$, corresponds to the deterministic drift of the forward GOU process run in the opposite direction. The second term, involving the score function $\nabla_x \log p_{T-t}(\cdot)$, compensates for the stochastic diffusion introduced in the forward process and guides the trajectories toward regions of high probability under the evolving distribution. 
Together, these dynamics transform samples from a stationary  distribution back into structured data while preserving the geometry induced by the matrix $M$.

\subsection{Score approximation}

The reverse-time dynamics \eqref{eq:reverse-sde-gou} depend on the score
function $\nabla_x \log p_t(x_t)$, gradient of the log density of marginals with respect to the input $x$. In
practice, this quantity is unknown and need to be approximated from data.
Following the framework of score-based generative modeling
\citep{song2019generative}, we introduce a neural network
$s_\theta(x,t)$ that approximates  scores:
\[
s_\theta(x,t) \approx \nabla_x \log p_t(x)\,.
\]

The network is trained using samples generated by the forward diffusion
process so that it learns the gradient of the log-density at different
noise levels. Once trained, the learned score network is substituted into the
reverse-time dynamics \eqref{eq:reverse-sde-gou}, yielding the practical
sampling process
\[
\dd \bar{X}_t
=
\big[
- M (\bar{X}_t - \mu)
+ \beta s_\theta(\bar{X}_t,T-t)
\big] \dd t
+
\sqrt{\beta}\, \dd \bar{W}_t\,.
\]

This reverse process enables the generation of samples by
progressively transforming a base distribution into structured data while preserving
the geometry encoded by the drift matrix $M$.

\subsection{Initialization of the backward process}

Unlike standard isotropic diffusion, which converges to a stationary distribution given by a standard Gaussian, the GOU forward process does not admit such a simple limiting distribution. Consequently, the backward generative process cannot be initialized from standard Gaussian noise.

In principle, the correct initialization for the reverse-time SDE is given by the distribution of the forward process at the terminal time $T$  using the expressions for the mean $\mu_T$ and covariance $\Sigma_T$ derived in Theorem~\ref{thm:mean} and Theorem~\ref{thm:covariance}. An effective approximation  is to initialize the backward process from a mixture distribution constructed across a set of structural anchors (e.g., cluster centers or representative data samples). Specifically, we define the practical initialization distribution $\hat{p}_T$ by assigning a localized Gaussian profile to each anchor:
\begin{equation}\label{eq:init}
    \hat{p}_T(x) = \mathbb{E}_{\hat{X}_0 \sim p_{\text{anchor}}}\left[ \mathcal{N}\big(x; \mu + e^{MT}(\hat{X}_0 - \mu), \Sigma_T\big) \right],
\end{equation}
where $p_{\text{anchor}}$ is the empirical distribution over the chosen structural anchors $\{\hat{X}_0^{(k)}\}_{k=1}^K$. To sample from $\hat{p}_T$ at the start of generation, one simply selects an anchor $\hat{X}_0$ (e.g., a cluster center where centers are obtained via a clustering algorithm) at random and draws a sample from its corresponding shifted Gaussian component. 
Recall that  $\Sigma_T$ only needs to be computed once via~\eqref{eq:covariance-solution}.
This initialization ensures that sample generation directly originates from a multimodal base distribution tailored to the target data geometry. Our empirical observations indicate that this approach performs remarkably well in practice.
Below, we study the initialization error induced employing the initialization scheme in~\eqref{eq:init}.

\begin{lemma}[Wasserstein bound on the initialization error]\label{lem:wass-init-error}
Let $p_T$ be the true distribution of the forward GOU process at terminal time $T$ initialized from the data distribution $X_0 \sim p_0$. Let $\hat{p}_T$ be the practical initialization distribution of the backward process defined in~\eqref{eq:init}. 
Let $\pi(X_0, \hat{X}_0)$ be the joint distribution (coupling) matching data points to their corresponding structural anchors under the data generation scheme. Then, we have
\begin{equation}
    \mathcal{W}_2(p_T, \hat{p}_T) \le  e^{c_{\max} T} \sqrt{\mathbb{E}_{(X_0, \hat{X}_0) \sim \pi}\big[\|X_0 - \hat{X}_0\|^2\big]}\,,
\end{equation}
where $c_{\max} \coloneqq \max_{i} c_i$ is the maximum (least negative) eigenvalue of the contraction drift matrix $M$.
\end{lemma}

\begin{proof}
By  definition, $\mathcal{W}_2^2(p_T, \hat{p}_T)$ is defined as the infimum of the expected squared Euclidean distance over all valid joint distributions whose marginals match $p_T$ and $\hat{p}_T$. We upper-bound this distance by constructing a specific synchronous coupling between the true forward trajectory $X_T$ and the generative initialization trajectory $\hat{X}_T$, where the pair of initial states $(X_0, \hat{X}_0)$ is drawn from the anchor assignment distribution $\pi$.

Recalling the SDE definition in \eqref{eq:structured-sde}, the exact solution for the true terminal state $X_T$ given an initial data point $X_0$ is:
\begin{equation}\label{eq:true-XT}
    X_T = \mu + e^{MT}(X_0 - \mu) + \sqrt{\beta} \int_0^T e^{M(T-s)} \, \mathrm{d}W_s\,.
\end{equation}
We initialize the approximate trajectory using the corresponding anchor $\hat{X}_0$ subjected to the identical realization of the Brownian motion $W_s$:
\begin{equation}\label{eq:approx-XT}
    \hat{X}_T = \mu + e^{MT}(\hat{X}_0 - \mu) + \sqrt{\beta} \int_0^T e^{M(T-s)} \, \mathrm{d}W_s\,.
\end{equation}
Since the marginal distributions of $X_T$ and $\hat{X}_T$ under this coupling are  $p_T$ and $\hat{p}_T$ respectively, we have:
\begin{equation*}
    \mathcal{W}_2^2(p_T, \hat{p}_T) \le \mathbb{E}_{(X_0, \hat{X}_0) \sim \pi}\big[ \| X_T - \hat{X}_T \|^2 \big]\,.
\end{equation*}
Subtracting \eqref{eq:approx-XT} from \eqref{eq:true-XT}, the stochastic integral terms cancel out identically, yielding a purely deterministic structural error expression for the coupled states:
\begin{equation*}
    X_T - \hat{X}_T = e^{MT}(X_0 - \hat{X}_0)\,.
\end{equation*}
Taking the expectation over the joint distribution $\pi(X_0, \hat{X}_0)$ of the squared Euclidean norm on both sides, we obtain:
\begin{equation*}
    \mathbb{E}_{(X_0, \hat{X}_0) \sim \pi}\big[ \| X_T - \hat{X}_T \|^2 \big] = \mathbb{E}_{(X_0, \hat{X}_0) \sim \pi}\big[ \| e^{MT}(X_0 - \hat{X}_0) \|^2 \big]\,.
\end{equation*}
To explicitly extract the final bound involving $c_{\max}$, we evaluate the matrix exponential using the spectral decomposition of the symmetric drift matrix $M = U \operatorname{diag}(c_1, \dots, c_d) U^\top$, where $U$ is an orthogonal matrix preserving the Euclidean norm:
\begin{align*}
    \mathbb{E}_{(X_0, \hat{X}_0) \sim \pi}\big[ \| e^{MT}(X_0 - \hat{X}_0) \|^2 \big] &= \mathbb{E}_{(X_0, \hat{X}_0) \sim \pi}\big[ \| U \operatorname{diag}\big(e^{c_1 T}, \dots, e^{c_d T}\big) U^\top (X_0 - \hat{X}_0) \|^2 \big] \\
    &= \sum_{i=1}^d e^{2c_i T} \mathbb{E}_{(X_0, \hat{X}_0) \sim \pi}\big[ |(U^\top X_0)_i - (U^\top \hat{X}_0)_i|^2 \big]\,.
\end{align*}
Since  $c_i < 0$ for all directions, meaning the exponential decay is strictly bounded above  by the maximum (least negative) eigenvalue $c_{\max} = \max_i c_i$. Factoring this term out of the summation yields:
\begin{align*}
    \sum_{i=1}^d e^{2c_i T} \mathbb{E}_{(X_0, \hat{X}_0) \sim \pi}\big[ |(U^\top X_0)_i - (U^\top \hat{X}_0)_i|^2 \big] &\le e^{2c_{\max} T} \sum_{i=1}^d \mathbb{E}_{(X_0, \hat{X}_0) \sim \pi}\big[ |(U^\top X_0)_i - (U^\top \hat{X}_0)_i|^2 \big] \\
    &= e^{2c_{\max} T} \mathbb{E}_{(X_0, \hat{X}_0) \sim \pi}\big[ \| U^\top (X_0 - \hat{X}_0) \|^2 \big] \\
    &= e^{2c_{\max} T} \mathbb{E}_{(X_0, \hat{X}_0) \sim \pi}\big[ \| X_0 - \hat{X}_0 \|^2 \big]\,.
\end{align*}
Taking the square root on both sides directly produces the final contractive metric bound:
\begin{equation*}
    \mathcal{W}_2(p_T, \hat{p}_T) \le e^{c_{\max} T} \sqrt{\mathbb{E}_{(X_0, \hat{X}_0) \sim \pi}\big[ \| X_0 - \hat{X}_0 \|^2 \big]}\,.
\end{equation*}
This completes the proof.
\end{proof}

To see the practical importance of Lemma~\ref{lem:wass-init-error}, consider a multimodal target distribution partitioned into distinct clusters. Because $\hat{p}_T$ explicitly models these modes via the structural anchors, the mathematical definition of the $\mathcal{W}_2$-distance admits the following admissible coupling,  $\pi(X_0, \hat{X}_0)$ that pairs each true data point $X_0$ with its respective cluster center $\hat{X}_0 = \mu^{(k)}$. Under this assignment, the joint expectation $\mathbb{E}_{(X_0, \hat{X}_0) \sim \pi}\big[\|X_0 - \hat{X}_0\|^2\big]$ collapses into a weighted sum of local cluster variances. 

Consequently, the initialization error scales down proportionally to the local variance of each individual cluster, rather than the global variance of the entire dataset. This stands in stark contrast to classical isotropic diffusion models, where the initialization error is fundamentally scales with the second moment $\mathbb{E}\big[\|X_0\|^2\big]$~\citep[Lemma~16]{gao2024convergence}. By leveraging the GOU drift matrix $M$ paired with data-dependent structural anchors, our approach confines the initialization mismatch to tightly-bound local clusters, which then vanishes exponentially over time.

\begin{remark}[Hyperparameter selection for the drift matrix]
Lemma~\ref{lem:wass-init-error} provides guidance for selecting the hyperparameter $\gamma$ in~\eqref{eq:drift-matrix}. Since $c_{\max}<0$ determines the slowest contraction rate, choosing $\gamma$ too small means that $c_{\max}$ approaches zero, weakening the contraction and leading to a looser upper bound on the mismatch between $\hat{p}_T$ and $p_T$. Therefore, $\gamma$ should be chosen in a way that implies  contraction of the initialization error.
\end{remark}

\section{Discussion: advantages of structured drift over isotropic diffusion}\label{sec:disc}

The GOU framework, through its structured forward and backward process, offers several significant advantages over conventional isotropic diffusion models. These benefits stem directly from the geometry-aware drift matrix $M$, which encodes the covariance structure of the data.

\begin{enumerate}
   \item \textbf{Mode preservation.} By incorporating the data geometry through the drift matrix $M$, the diffusion process rapidly contracts low-variance directions toward gaussian noise while preserving separation along high-variance directions (Corollary~\ref{cor:modesep}). As a result, distinct modes remain distinguishable for longer time during the forward process, allowing the reverse process to recover them more faithfully. Moreover, the score matching objective effectively reduces to a low-dimensional problem, since the scores in directions that quickly become Gaussian are already well-known. In contrast, isotropic diffusion contracts all directions uniformly, causing premature mode mixing and potentially leading to mode collapse.
    \item \textbf{Preservation of correlation structure.} The anisotropic drift enables feature-specific evolution rates, allowing the process to retain important statistical dependencies encoded in the data covariance. Theorem~\ref{thm:covariance} shows that the covariance $\Sigma_t$ evolves as a combination of transported initial correlations and geometry-aware noise injection. Isotropic diffusion, by treating all dimensions identically, cannot distinguish between informative correlations and noise, potentially disrupting the very structure that generative models aim to learn.

    \item \textbf{Initialization error is governed by local variance.} While standard diffusion error scales with the overall spread of the dataset, GOU isolates the initial mismatch to localized variances around structural anchors. This structural alignment dramatically reduces initial error and ensures rapid contractive convergence. 

\end{enumerate}
These properties establish the GOU framework as a principled extension of diffusion-based generative modeling to structured data. 
By respecting the intrinsic geometry of the data distribution—rather than treating all directions uniformly—the method achieves better mode coverage, preserves statistical dependencies, improves initialization error,  and as a conjucture  converges under milder theoretical conditions. 
These advantages position GOU as particularly well-suited for scientific applications where correlations (e.g., between genes, SNPs, or physical measurements) carry essential information, and where mode collapse would compromise downstream analysis.

While our GOU process shares a similar motivation with Riemannian diffusion~\citep{de2022riemannian,huang2022riemannian}—namely, to incorporate geometric structure into the diffusion dynamics—the two approaches differ fundamentally in how this geometry is modeled. Riemannian diffusion introduces a position-dependent metric tensor, leading to locally adaptive and nonlinear diffusion that reflects the curvature of an underlying manifold. In contrast, our approach employs a fixed, data-driven linear operator derived from the covariance structure of the target distribution. As a result, the induced geometry is global rather than local, and the dynamics remain analytically tractable. Despite this simplification, the GOU process captures key aspects of anisotropic behavior by aligning the diffusion with principal directions of the data, thereby preserving important structural features such as mode separation. In this sense, our method can be viewed as a linear and computationally efficient approximation to more general geometry-aware diffusion processes.

\section{Empirical studies}\label{sec:sim}
We evaluate the proposed GOU process on both synthetic and real-world datasets. As motivated from the outset, our application domain focuses on biologically structured data, where isotropic drifts fail to adequately capture the underlying structure of the data. We compare the proposed GOU process against an isotropic-drift baseline.
For all experiments, both methods are trained using identical neural score network architectures and hyperparameters to ensure a fair comparison. Specifically, we use a multilayer perceptron (MLP) with four hidden layers of width 512, SiLU activations, and a learned 64-dimensional time embedding concatenated to the input. For each dataset, we generate synthetic samples from both models and evaluate distributional fidelity using: (i) first- and second-moment errors (differences in the mean and standard deviation), (ii) covariance preservation measured by the Frobenius norm of the difference between covariance matrices, and (iii) distributional distances, including the sliced Wasserstein-2 distance and the Fréchet Distance (FD).

\subsection{Mixture of two Gaussians}
In this experiment, we compare the performance of the GOU process versus isotropic drift on a two-dimensional mixture of Gaussians with separated modes. We first generate a dataset of $N=2000$ samples from a mixture of two Gaussian components and simulate forward processes for both approaches (see Figure~\ref{fig:mode-separation} for a visual presentation). For each method, we train a neural score network to approximate the score function and perform reverse sampling to generate synthetic data. 
The generated samples are quantitatively evaluated  against true data samoples using several metrics. Results are summarized in Table~\ref{tab:sim}, showing that the GOU process consistently produces samples closer to the original distribution across all metrics compared to the isotropic drift.
\begin{table}[h!]
\centering
\caption{Comparison of isotropic and structured drifts for a simulated two-dimensional mixture of Gaussians}\label{tab:sim}
\begin{tabular}{l S[table-format=1.4] S[table-format=1.4]}
\hline
\textbf{Metric}  & \textbf{Isotropic} & \textbf{Structured} \\
\hline
Mean Error      & 0.05 & \textbf{0.03} \\
Std Error       &0.09 & \textbf{0.07} \\
Covariance Frobenius Norm& 0.13 & \textbf{0.11} \\
Wasserstein Distance & 0.07 & \textbf{0.06} \\
\hline
\end{tabular}
\end{table}

\subsection{Synthetic LD-like simulation}

To evaluate the models on data with dependency patterns, we generated a synthetic genetic dataset consisting of $N=500$ individuals and $d=1000$ single-nucleotide polymorphisms (SNPs).
Note that this simulation setup reflects the high-dimensional regime commonly encountered in genetic studies, where the number of SNPs substantially exceeds the number of samples. The SNPs were simulated to mimic the local dependency structure observed in real genomic data, where nearby variants exhibit stronger correlations due to linkage disequilibrium (LD). Specifically, the SNPs were partitioned into blocks of fixed size $20$, and within each block correlated latent Gaussian variables were generated using an exponentially decaying covariance structure (decay rate $0.5$). These latent variables were then discretized to obtain genotype values in $\{0,1,2\}$, corresponding to the number of minor alleles carried by an individual. This procedure produces a dataset with realistic marginal distributions and local correlation patterns similar to those observed in genomic studies. 
We then evaluate the performance of GOU vs isotropic diffusion models.  
Results are reported in Table~\ref{tab:genetic_sim} and demonstrate that the structured drift substantially outperforms the isotropic drift across all metrics except the mean error.
Furthermore, the correlation heatmaps in Figure~\ref{fig:coregen} illustrate that the GOU model better preserves the correlation structure of the original genotype data, whereas the isotropic diffusion fails to recover these dependencies. Overall, these findings highlight the importance of incorporating structured covariance information when modeling correlated high-dimensional biological data using score-based generative methods.

\begin{table}[h!]
\centering
\caption{Comparison of isotropic and structured drift  on simulated SNP data }
\label{tab:genetic_sim}

\begin{tabular}{l S[table-format=1.3] S[table-format=1.3]}
\hline
Metric & \textbf{Isotropic} & \textbf{Structured} \\
\hline
Mean Error & \textbf{0.017} & 0.019 \\
Std Error & 0.08 & \textbf{0.05} \\
Covariance Frobenius Norm & 4.09 & \textbf{3.71} \\
Wasserstein Distance  & 0.04 & \textbf{0.03} \\
FD & 2.04 & \textbf{0.88}\\
\hline
\end{tabular}
\end{table}
\begin{figure}
    \centering
    \includegraphics[width=0.8\linewidth]{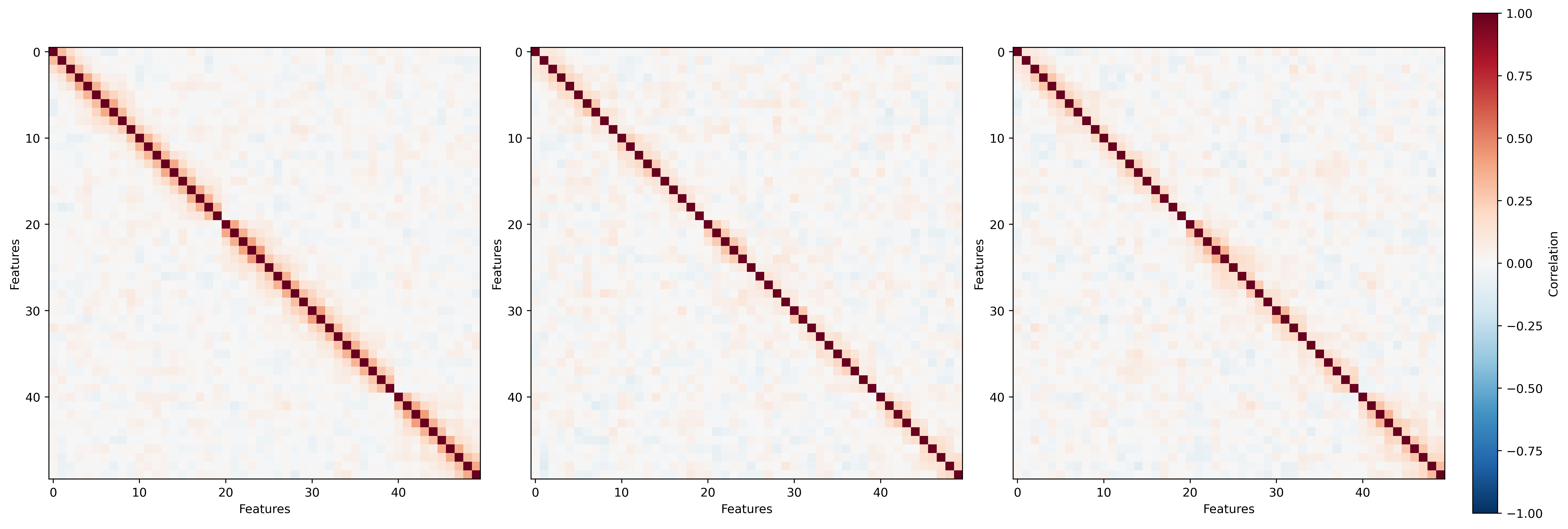}
    \caption{SNP structure preservation. (Left–Right) Correlation heatmaps for ground truth, isotropic diffusion, and the proposed GOU process on simulated genetic data (first 100 SNPs; 20-SNP block gridlines). The proposed framework  more accurately reproduces block boundaries and correlation magnitudes.}
    \label{fig:coregen}
\end{figure}

\subsection{Single-cell PBMC transcriptomic dataset}
We evaluate our method on the PBMC3K dataset~\citep{zheng2017massively}, which contains approximately 2,700 single-cell RNA sequencing profiles across more than 13,000 genes. 
For simplicity, we restrict our analysis to the first 500 genes.
The data exhibits clear heterogeneity across multiple cell types (including T cells, B cells, NK cells, and monocytes) and a moderate gene–gene correlation structure, making it a standard benchmark for evaluating generative models on high-dimensional and structured data.
Table~\ref{tab:evaluation_results_PBMC} and Figure~\ref{fig:coregenPBMC} summarize our experimental results. The proposed structured drift model consistently outperforms the identity baseline across all evaluation metrics. In particular, it better preserves  gene–gene interaction structure, suggesting that incorporating data-driven drift improves the fidelity of generated samples. These results highlight the importance of matching the generative model structure to the intrinsic correlations in the data when modeling high-dimensional and structured data.

\begin{table}[h!]
\centering
\caption{Comparison of isotropic and structured drift  on PBMC3K dataset }
\label{tab:evaluation_results_PBMC}

\begin{tabular}{l
S[table-format=2.4]
S[table-format=2.4]}
\hline
\textbf{Metric} & \textbf{Isotropic} & \textbf{Structured} \\
\hline

Mean Error & 0.03 & \textbf{0.02} \\
Std Error & 0.22 & \textbf{0.13} \\
Covariance Frobenius Norm & 7.42 & \textbf{4.72} \\
Wasserstein distance & 0.14 & \textbf{0.07} \\
FD & 6.62 &  \textbf{1.99}\\
\hline
\end{tabular}
\end{table}

\begin{figure}
    \centering
    \includegraphics[width=0.8\linewidth]{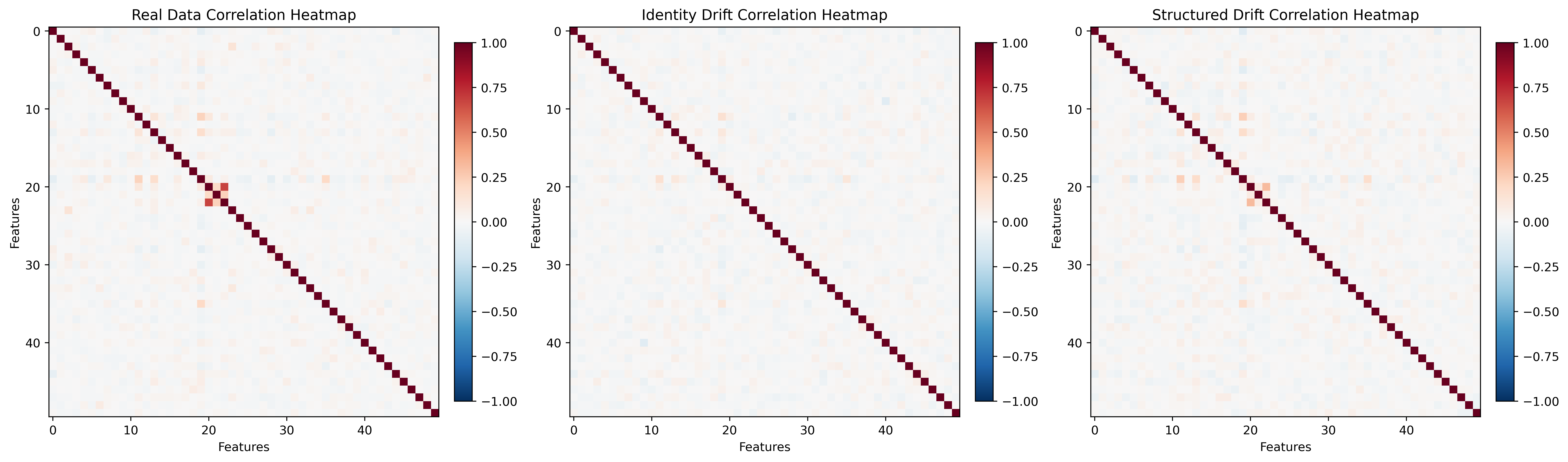}
    \caption{Gene–gene correlation structure preservation in PBMC3K dataset. (Left–Right) Correlation heatmaps for ground truth, isotropic diffusion, and the proposed GOU model on PBMC3K single-cell RNA-seq dataset (first 50 genes). The isotropic baseline fails to capture structured gene–gene dependencies induced by heterogeneous cell populations, whereas the proposed model more accurately preserves correlation patterns and recovers biologically meaningful co-expression structure.}
    \label{fig:coregenPBMC}
\end{figure}

\subsection{Single-cell RNA-seq Paul15 dataset}
We evaluate our method on the Paul15 dataset~\citep{paul2015}, which consists of single-cell RNA sequencing profiles capturing mouse hematopoietic stem cell differentiation across a continuous developmental trajectory. The dataset contains expression measurements for  $2{,}700$ cells and over $5{,}000$ genes, and exhibits a well-characterized branching structure, where cells progressively differentiate into distinct blood lineages, making it a standard benchmark for trajectory inference and generative modeling of dynamic biological processes.
For simplicity, we restrict our analysis to the first 500 genes.
Table~\ref{tab:evaluation_results_Paul15} and Figure~\ref{fig:coregenPaul} summarize our experimental results. 
Our proposed method accurately recovers the underlying structure and preserves biologically meaningful gene expression patterns along the differentiation trajectory, whereas the isotropic drift approach completely fails to recover the meaningful correlation structure present in these highly structured data.
This observation highlights the necessity of geometry-aware approaches for modeling highly structured data, where isotropic methods fail to capture the underlying geometric and correlation structure.

\begin{table}[h!]
\centering
\caption{Comparison of isotropic and structured drift on Paul15 dataset }
\label{tab:evaluation_results_Paul15}

\begin{tabular}{l
S[table-format=3.4]
S[table-format=2.4]}
\hline
\textbf{Metric} & \textbf{Isotropic} & \textbf{Structured} \\
\hline

Mean Error & 0.02 & \textbf{0.01} \\
Std Error & 0.87 & \textbf{0.84} \\
Covariance Frobenius Norm & 391.40 & \textbf{9.45} \\
Wasserstein distance & 0.59 & \textbf{0.08} \\
FD & 326.25 &  \textbf{1.35}\\
\hline
\end{tabular}
\end{table}

\begin{figure}
    \centering
    \includegraphics[width=0.8\linewidth]{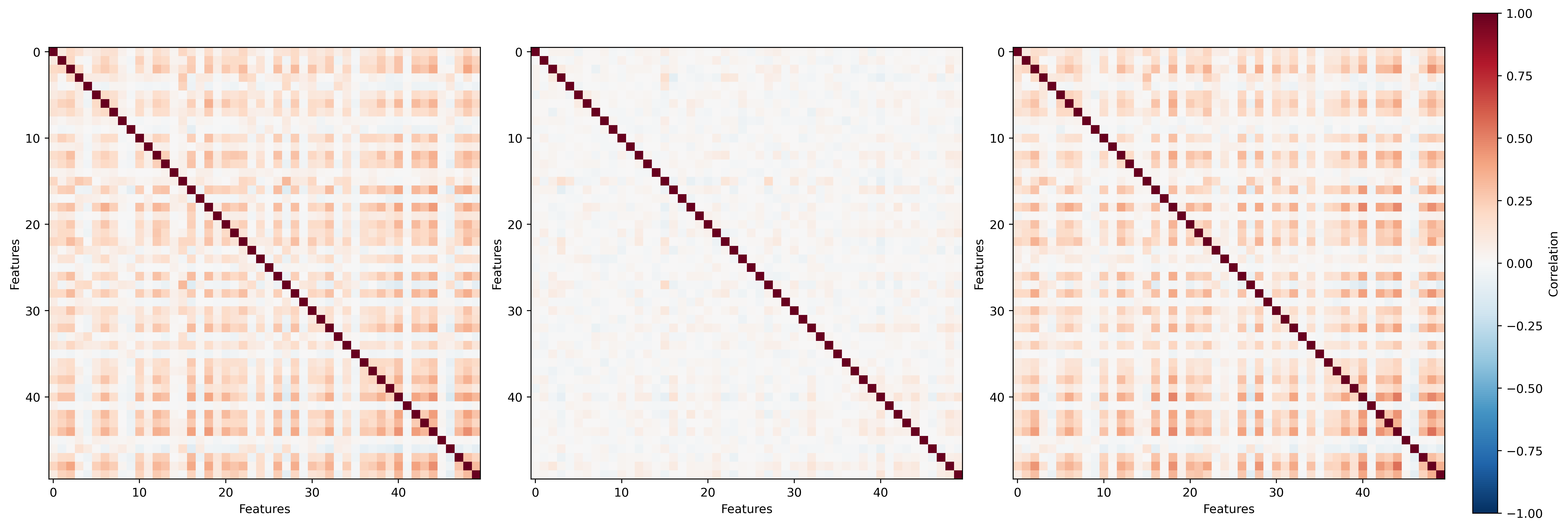}
    \caption{Gene–gene correlation structure preservation in Paul15 dataset. (Left–Right) Correlation heatmaps for ground truth, isotropic diffusion, and the proposed GOU process on the Paul15 single-cell RNA-seq dataset (first 50 genes). While the isotropic baseline fails to capture the strong gene–gene dependencies, the proposed method   accurately preserves correlation patterns and recovers biologically meaningful co-expression structure consistent with the underlying developmental trajectory.}
    \label{fig:coregenPaul}
\end{figure}

\section*{Conclusion}
In this work, we introduced Geometry-aware Ornstein-Uhlenbeck diffusion models, which incorporate the intrinsic geometric structure of data into the forward diffusion process. By designing a variance-dependent drift matrix, the GOU process contracts low-variance directions quickly while preserving high-variance directions, effectively preventing mode merging in multimodal data. We derived explicit solutions for the evolution of both the mean and covariance, showing how the process transports correlations and injects noise in a geometry-aware manner. The corresponding reverse-time dynamics provide a principled generative mechanism that can reconstruct structured data while respecting the underlying geometry. Overall, GOU offers a simple yet powerful framework to integrate data-dependent structure into diffusion models, bridging the gap between classical isotropic diffusion and geometry-preserving generative modeling.

\emph{Future work.} While existing convergence analyses for reverse-time diffusion processes often rely on global (weak) log-concavity assumptions~\citep{block2020generative,lee2023convergence,gao2025wasserstein,kremling2025non}, such conditions are rarely satisfied by complex real-world data. The geometry-aware structure of the GOU process suggests that substantially weaker assumptions, such as local log-concavity or monotonicity, may already be sufficient to guarantee global contraction in Wasserstein distance. Establishing such non-asymptotic guarantees for multimodal and heavy-tailed distributions remains an important direction for future work.

\acks{M. Taheri is grateful for partial funding by the Deutsche Forschungsgemeinschaft (DFG,
German Research Foundation) under project numbers 541176257 and 520388526 (TRR391).}

\appendix
\section{Technical results and proofs}
Here, we first present a technical result that will be used in the proofs. The proofs of the main results then follow.
\begin{lemma}[Matrix exponential of a diagonalizable matrix]\label{lem:Mex}
Let $M \in \mathbb{R}^{d \times d}$ be a symmetric matrix with eigendecomposition
\[
M = U \operatorname{diag}(c_1,\dots,c_d)\, U^\top,
\]
where $U$ is orthogonal and $c_1,\dots,c_d \in \mathbb{R}$. Then the matrix exponential satisfies
\[
e^{Mt} = U \operatorname{diag}(e^{c_1 t}, \dots, e^{c_d t})\, U^\top.
\]
\end{lemma}
\subsection{Proof of Lemma~\ref{lem:Mex}}
\begin{proof}
By definition of the matrix exponential,
\[
e^{Mt} := \sum_{k=0}^{\infty} \frac{(Mt)^k}{k!}
= \sum_{k=0}^{\infty} \frac{t^k M^k}{k!}.
\]

We first show that
\[
M^k = U \operatorname{diag}(c_1^k,\dots,c_d^k)\, U^\top.
\]

We proceed by induction. For $k=1$, the statement is trivial. Assume it holds for some $k \ge 1$. Then
\[
M^{k+1}
= M^k M
= \left(U \Lambda^k U^\top\right)\left(U \Lambda U^\top\right)
= U \Lambda^k (U^\top U)\Lambda U^\top
= U \Lambda^{k+1} U^\top,
\]
where $\Lambda = \operatorname{diag}(c_1,\dots,c_d)$ and we used $U^\top U = I$. Hence the claim holds for all $k \ge 0$.

Substituting into the series expansion yields
\[
e^{Mt}
= \sum_{k=0}^{\infty} \frac{t^k}{k!} U \Lambda^k U^\top
= U \left(\sum_{k=0}^{\infty} \frac{t^k \Lambda^k}{k!}\right) U^\top.
\]

Since $\Lambda$ is diagonal, powers act entrywise:
\[
\sum_{k=0}^{\infty} \frac{t^k \Lambda^k}{k!}
= \operatorname{diag}\!\left(\sum_{k=0}^{\infty} \frac{(c_1 t)^k}{k!}, \dots, \sum_{k=0}^{\infty} \frac{(c_d t)^k}{k!}\right).
\]

Each scalar series equals the exponential function, i.e.,
\[
\sum_{k=0}^{\infty} \frac{(c_i t)^k}{k!} = e^{c_i t}.
\]

Therefore,
\[
e^{Mt} = U \operatorname{diag}(e^{c_1 t}, \dots, e^{c_d t}) U^\top.
\]
\end{proof}












\subsection{Proof of Theorem~\ref{thm:covariance}}\label{sec:poofcov}
\begin{proof}
We do the proof in four steps: 

\emph{Step 1: Centering the process.}

Let $\mu_t = \mathbb{E}[X_t]$ and define the centered process
\[
\bar X_t := X_t - \mu_t\,.
\]
By construction, $\mathbb{E}[\bar X_t] = 0$.

Using the SDE for $X_t$ and the dynamics of $\mu_t$ (similar approach as in the proof of Theorem~\ref{thm:mean}), we compute
\[
\dd \bar X_t = \dd X_t - \dd \mu_t\,.
\]
Substituting,
\begin{align*}
\dd \bar X_t
&= M(X_t - \mu)\,\dd t + \sqrt{\beta}\,\dd W_t - M(\mu_t - \mu)\,\dd t \\
&= M\big[(X_t - \mu) - (\mu_t - \mu)\big]\,\dd t + \sqrt{\beta}\,\dd W_t \\
&= M \bar X_t\,\dd t + \sqrt{\beta}\,\dd W_t\,.
\end{align*}

\medskip

\emph{Step 2: Applying Itô's formula.}

We now study the matrix-valued process $\bar X_t \bar X_t^\top$. By Itô's product rule,
\[
\dd (\bar X_t \bar X_t^\top)
= \dd \bar X_t \, \bar X_t^\top
+ \bar X_t \, \dd \bar X_t^\top
+ \dd \bar X_t \, \dd \bar X_t^\top\,.
\]

We compute each term separately.

\medskip


Using $\dd \bar X_t = M \bar X_t\,\dd t + \sqrt{\beta}\,\dd W_t$, we obtain
\begin{align*}
\dd \bar X_t \, \bar X_t^\top
&= (M \bar X_t\,\dd t + \sqrt{\beta}\,\dd W_t)\,\bar X_t^\top \\
&= M \bar X_t \bar X_t^\top \, \dd t
+ \sqrt{\beta}\, \dd W_t \, \bar X_t^\top\,,
\end{align*}
and similarly,
\begin{align*}
\bar X_t \, \dd \bar X_t^\top
&= \bar X_t ( \bar X_t^\top M^\top \dd t + \sqrt{\beta}\, \dd W_t^\top ) \\
&= \bar X_t \bar X_t^\top M^\top \dd t
+ \sqrt{\beta}\, \bar X_t \dd W_t^\top\,.
\end{align*}

We now compute the quadratic variation term $\dd \bar X_t\, \dd \bar X_t^\top$ explicitly.

First, expand the product:
\begin{align*}
\dd \bar X_t\, \dd \bar X_t^\top
&= (M \bar X_t\, \dd t + \sqrt{\beta}\, \dd W_t)
   (M \bar X_t\, \dd t + \sqrt{\beta}\, \dd W_t)^\top \\
&= (M \bar X_t\, \dd t)(M \bar X_t\, \dd t)^\top \\
&\quad + (M \bar X_t\, \dd t)(\sqrt{\beta}\, \dd W_t)^\top \\
&\quad + (\sqrt{\beta}\, \dd W_t)(M \bar X_t\, \dd t)^\top \\
&\quad + (\sqrt{\beta}\, \dd W_t)(\sqrt{\beta}\, \dd W_t)^\top\,.
\end{align*}

We now analyze each term separately using standard Itô rules.

\medskip

\textit{(i) Drift--drift term.}
\[
(M \bar X_t\, \dd t)(M \bar X_t\, \dd t)^\top
= M \bar X_t \bar X_t^\top M^\top \, \dd t^2\,.
\]
Since $\dd t^2 = 0$ in Itô calculus, this term vanishes.

\medskip

\textit{(ii) Drift--noise term.}
\[
(M \bar X_t\, \dd t)(\sqrt{\beta}\, \dd W_t)^\top
= \sqrt{\beta}\, M \bar X_t\, \dd t\, \dd W_t^\top\,.
\]
Since $\dd t\, \dd W_t = 0$, this term vanishes.

\medskip

\textit{(iii) Noise--drift term.}
\[
(\sqrt{\beta}\, \dd W_t)(M \bar X_t\, \dd t)^\top
= \sqrt{\beta}\, \dd W_t\, \dd t\, \bar X_t^\top M^\top\,.
\]
Since $\dd W_t\, \dd t = 0$, this term also vanishes.

\medskip

\textit{(iv) Noise--noise term.}
\[
(\sqrt{\beta}\, \dd W_t)(\sqrt{\beta}\, \dd W_t)^\top
= \beta\, \dd W_t \dd W_t^\top\,.
\]
Using the quadratic variation identity of Brownian motion,
\[
\dd W_t \dd W_t^\top = \mathbb{I}_d\, \dd t\,,
\]
we obtain
\[
(\sqrt{\beta}\, \dd W_t)(\sqrt{\beta}\, \dd W_t)^\top
= \beta \mathbb{I}_d\, \dd t\,.
\]

\medskip

Combining all four terms yields
\[
\dd \bar X_t\, \dd \bar X_t^\top = \beta \mathbb{I}_d\, \dd t\,.
\]

\medskip


Summing all contributions,
\begin{align*}
\dd(\bar X_t \bar X_t^\top)
&= M \bar X_t \bar X_t^\top \dd t
+ \bar X_t \bar X_t^\top M^\top \dd t \\
&\quad + \sqrt{\beta}\, \dd W_t \bar X_t^\top
+ \sqrt{\beta}\, \bar X_t \dd W_t^\top
+ \beta \mathbb{I}_d \dd t\,.
\end{align*}

\medskip

\emph{Step 3: Taking expectations.}

Define the covariance matrix
\[
\Sigma_t := \mathbb{E}[\bar X_t \bar X_t^\top]\,.
\]
Taking expectations and using $\mathbb{E}[\dd W_t] = 0$, the stochastic integral terms vanish:
\[
\mathbb{E}[\dd W_t \bar X_t^\top] = 0,
\quad
\mathbb{E}[\bar X_t \dd W_t^\top] = 0\,.
\]
Hence,
\[
\mathbb{E}\Bigl[ \dd (\bar X_t \bar X_t^\top )\Bigr]
= M \Sigma_t \dd t + \Sigma_t M^\top \dd t + \beta \mathbb{I}_d \dd t\,.
\]

Since $\bar X_t$ has finite second moments, we can interchange expectation and differentiation, yielding
\[
\frac{\dd}{\dd t} \Sigma_t
= M \Sigma_t + \Sigma_t M^\top + \beta \mathbb{I}_d,
\qquad
\Sigma_0 = \operatorname{Cov}(X_0)\,.
\]

\medskip

\emph{Step 4: Solving the matrix differential equation.}

Define
\[
A_t := e^{-Mt} \Sigma_t e^{-M^\top t}\,.
\]
We differentiate using the product rule:
\begin{align*}
\frac{\dd}{\dd t} A_t
&= (-M e^{-Mt}) \Sigma_t e^{-M^\top t}
+ e^{-Mt} \dot\Sigma_t e^{-M^\top t}
+ e^{-Mt} \Sigma_t (-M^\top e^{-M^\top t}) \\
&= e^{-Mt} \left( \dot\Sigma_t - M\Sigma_t - \Sigma_t M^\top \right) e^{-M^\top t}\,.
\end{align*}

Substituting the differential equation for $\Sigma_t$,
\[
\dot\Sigma_t - M\Sigma_t - \Sigma_t M^\top = \beta \mathbb{I}_d\,,
\]
we obtain
\[
\frac{\dd}{\dd t} A_t = \beta e^{-Mt} e^{-M^\top t}\,.
\]



Integrating from $0$ to $t$,
\[
A_t = A_0 + \beta \int_0^t e^{-Ms} e^{-M^\top s} \dd s\,.
\]
Since $A_0 = \Sigma_0$, this gives
\[
A_t = \Sigma_0 + \beta \int_0^t e^{-Ms} e^{-M^\top s} \dd s\,.
\]

\medskip


Multiplying from the left by $e^{Mt}$ and from the right by $e^{M^\top t}$ yields
\[
\Sigma_t
= e^{Mt} \Sigma_0 e^{M^\top t}
+ \beta \int_0^t e^{M(t-s)} e^{M^\top (t-s)} \dd s\,,
\]
which is the desired result.

\end{proof}

\subsection{Proof of Lemma~\ref{th:SCov}}\label{sec:proofSCov}
\begin{proof}
From Theorem~\ref{thm:covariance}, we have
\begin{equation*}
\Sigma_t = e^{Mt}\Sigma_0 e^{M^\top t} + \beta \int_0^t e^{M(t-s)} e^{M^\top(t-s)}\,\dd s\,.
\end{equation*}

We analyze the limit $t \to \infty$.

\medskip
\noindent\emph{Step 1: Decay of the initial condition.}

Since $M$ is negative definite, all eigenvalues have strictly negative real parts, hence $\lim_{t \to \infty} e^{Mt}$ goes to zero matrix, 
which implies
\[
\lim_{t \to \infty} e^{Mt}\Sigma_0 e^{M^\top t} = 0_{d\times d}\,.
\]

\medskip
\noindent\emph{Step 2: Limit of the noise term.}

Define
\[
\Sigma_\infty := \beta \int_0^\infty e^{Mu} e^{M^\top u}\,\dd u\,.
\]

The integral is well-defined because $e^{Mu}$ decays exponentially fast as $u \to \infty$.

\medskip
\noindent\emph{Step 3: Lyapunov equation.}

Let
\[
J := \int_0^\infty e^{Mu} e^{M^\top u}\,\dd u\,.
\]

We compute
\[
\frac{\dd}{\dd u}\big(e^{Mu} e^{M^\top u}\big)
= M e^{Mu} e^{M^\top u} + e^{Mu} e^{M^\top u} M^\top\,.
\]

Integrating over $[0,\infty)$ yields
\[
\big[e^{Mu} e^{M^\top u}\big]_0^\infty
= M J + J M^\top\,.
\]

Since $e^{Mu} e^{M^\top u} \to 0$ as $u \to \infty$ and equals $\mathbb{I}_d$ at $u=0$, we obtain

\begin{equation}
M J + J M^\top = -\mathbb{I}_d\,.
\label{eq:lyap}
\end{equation}

This is a continuous-time Lyapunov equation. Because $M$ is negative definite, the solution $J$ is unique.

\medskip
\noindent\emph{Step 4: Identification of the stationary covariance.}

We have
\[
\Sigma_\infty = \beta J\,,
\]
and multiplying \eqref{eq:lyap} by $\beta$ gives
\[
M \Sigma_\infty + \Sigma_\infty M^\top + \beta \mathbb{I}_d = 0\,.
\]

\medskip
\noindent\emph{Step 5: Symmetric case.}

For $M$ symmetric, then $M = M^\top$ and the Lyapunov equation becomes
\[
M \Sigma_\infty + \Sigma_\infty M + \beta \mathbb{I}_d = 0\,.
\]

Since $M$ and $\Sigma_\infty$ commute in this case (they share eigenvectors), we get
\[
2M \Sigma_\infty + \beta \mathbb{I}_d = 0\,,
\]
hence
\[
\Sigma_\infty = -\frac{\beta}{2} M^{-1}\,.
\]

This completes the proof.
\end{proof}

\vskip 0.2in
\bibliography{sample}

\end{document}